\documentclass{article}
\usepackage[cp1251]{inputenc}
\usepackage[russian]{babel}
\usepackage{amssymb,amsfonts,amsmath,amsthm,amscd}
\usepackage{graphicx}
\usepackage{enumerate}
\usepackage{soul}
\usepackage[matrix,arrow,curve]{xy}
\usepackage{longtable}
\usepackage{array}
\usepackage{bm}
\RequirePackage{russlh}
\RequirePackage{mathlh}

\newdimen\symskip
\newdimen\defskip
\newdimen\parind
\parind=\parindent
\newdimen\leftmarge
\newdimen\theoremshape
\topsep\defskip
\makeatletter
\newcommand*{\clei}{\nobreak\hskip\z@skip}

\renewcommand{\:}{\textup{:}}

\DeclareRobustCommand*{\д}{\clei\hbox{-}\clei}
\newcommand{\no}{}
\renewcommand{\@listI}{\settowidth\labelwidth{\labheadi{\no}}\listipar{\parind}{\labelwidth}}
\newcommand{\listivpar}{\topsep\defskip\partopsep0pt\parsep-\parskip\itemsep0.5\topsep}
\newcommand{\listipar}[2]{\rightmargin0pt\leftmargin#1\labelsep#1\advance\labelsep-#2\itemindent0pt\listivpar}
\renewcommand{\@listii}{\settowidth\labelwidth{\labheadii{\@roman{\no}}}\listiipar{\parind}{\labelwidth}}
\newcommand{\listiivpar}{\topsep0.5\defskip\partopsep0pt\parsep-\parskip\itemsep0.5\topsep}
\newcommand{\listiipar}[2]{\rightmargin0pt\leftmargin#1\labelsep#1\advance\labelsep-#2\itemindent0pt\listiivpar}
\def\thempfn{\ifcase\value{footnote}1\or *\or **\or ***\else\@ctrerr\fi}
\renewcommand\footnoterule{%
  \kern-3\p@
  \hrule\@width1in
  \kern2.6\p@}
\makeatother

\makeatletter

\renewcommand{\@biblabel}[1]{[#1]}
\renewenvironment{thebibliography}[1]
     {\renewcommand{\refname}{References}%
      \renewcommand{\No}{}%
      \section*{\refname}%
      \@mkboth{\MakeUppercase\refname}{\MakeUppercase\refname}%
      \list{\@biblabel{\@arabic\c@enumiv}}%
           {\itemsep\baselineskip
            \leftmargin\parind
            \settowidth\labelwidth{\@biblabel{#1}}%
            \labelsep\parind\advance\labelsep-\labelwidth
            \@openbib@code
            \usecounter{enumiv}%
            \let\p@enumiv\@empty
            \renewcommand\theenumiv{\@arabic\c@enumiv}}%
      \sloppy
      \clubpenalty4000
      \@clubpenalty\clubpenalty
      \widowpenalty4000%
      \sfcode`\.\@m}
     {\def\@noitemerr
       {\@latex@warning{Empty `thebibliography' environment}}%
      \endlist}

\def\@maketitle{%
  \newpage
  \vskip0.5em%
  UDK \udk%
  \vskip0.5em%
  MSC \msc%
  \vskip1em%
  \begin{center}\bf%
  \let\footnote\thanks%
   {\Large\@author\par}%
   \vskip1.5em%
   {\LARGE\@title\par}%
   \vskip1em%
   {\large\@date}%
  \end{center}%
  \par
  \vskip1.5em}

\def\@title{\@latex@warning@no@line{No \noexpand\title given}}

\makeatother

\makeatletter

\renewcommand\sectionmark[1]{%
 \markright{%
  \ifnum \c@secnumdepth >\z@
   \thesection. \ %
  \fi
 #1}}%

\renewcommand{\section}{\@startsection{section}{1}{0pt}%
{5.5ex plus .5ex minus .2ex}{1.5ex plus .3ex}%
{\center\normalfont\Large\bfseries\sffamily\bom}}
\renewcommand{\subsection}{\@startsection{subsection}{2}{0pt}%
{4.5ex plus .4ex minus .2ex}{0.75ex plus .2ex}%
{\center\normalfont\large\bfseries\sffamily\bom}}
\renewcommand{\subsubsection}{\@startsection{subsubsection}{3}{0pt}%
{2.5ex plus .5ex minus .2ex}{1ex plus .2ex}%
{\center\normalfont\bfseries\sffamily\bom}}

\newcommand{\Ss}{\textup{\S\,}}

\def\@postskip@{\hskip.5em\relax}

\def\postsection{.\@postskip@}

\def\postsubsection{.\@postskip@}

\def\postsubsubsection{.\@postskip@}

\def\postparagraph{.\@postskip@}

\def\postsubparagraph{.\@postskip@}

\def\@seccntformat#1{\csname pre#1\endcsname\csname the#1\endcsname\csname post#1\endcsname}

\makeatother

\renewcommand{\thesection}{\textup{\arabic{section}}}

\newcommand{\parr}{\par\addvspace{\defskip}}
\newcommand{\theo}[2]{\newtheorem{#1}{#2}[section]}
\newcommand{\deff}[2]{\newenvironment{#1}{\parr\textbf{#2.}}{\parr}}
\theo{cas}{Case}
\theo{problem}{Problem}
\theo{theorem}{Theorem}
\theo{lemma}{Lemma}
\theo{prop}{Proposition}
\theo{stm}{Statement}
\theo{fact}{Fact}
\theo{imp}{Corollary}
\theo{ex}{Example}
\deff{df}{Definition}
\deff{note}{Note}
\deff{denote}{Notation}
\deff{denotes}{Notations}
\deff{hint}{Hint}
\deff{answer}{Answer\:}

\makeatletter
\def\@begintheorem#1#2[#3]{%
  \deferred@thm@head{\the\thm@headfont \thm@indent
    \@ifempty{#1}{\let\thmname\@gobble}{\let\thmname\@iden}%
    \@ifempty{#2}{\let\thmnumber\@gobble}{\let\thmnumber\@iden}%
    \@ifempty{#3}{\let\thmnote\@gobble}{\let\thmnote\@iden}%
    \thm@notefont{\bfseries\upshape}%
    \indent%
    \thm@swap\swappedhead\thmhead{#1}{#2}{#3}%
    \the\thm@headpunct
    \thmheadnl 
    \hskip\thm@headsep
  }%
  \ignorespaces}
\renewenvironment{proof}{\setcounter{cas}{0}\parr\pushQED{\qed}\normalfont$\square\quad$}{\setcounter{cas}{0}\popQED\@endpefalse\parr}

\makeatother

\newcommand{\labheadi}[1]{\textup{#1)}}
\newcommand{\labheadii}[1]{\textup{(#1)}}

\newenvironment{nums}[1]{\renewcommand{\no}{#1}\begin{enumerate}}{\end{enumerate}}
\newcommand{\eqn}[1]{\begin{equation}#1\end{equation}}
\newcommand{\equ}[1]{\begin{equation*}#1\end{equation*}}

\newcommand{\eqac}[1]{\equ{\begin{array}{c}#1\end{array}}}

\makeatletter

\def\LT@makecaption#1#2#3{%
  \LT@mcol\LT@cols c{\hbox to\z@{\hss\parbox[t]\LTcapwidth{%
    \sbox\@tempboxa{#1{#2. }#3}%
    \ifdim\wd\@tempboxa>\hsize
      #1{#2. }#3%
    \else
      \hbox to\hsize{\hfil\box\@tempboxa\hfil}%
    \fi
    \endgraf\vskip\baselineskip}%
  \hss}}}

\@addtoreset{equation}{section}
\@addtoreset{footnote}{section}

\newenvironment{casks}{%
  \matrix@check\casks\env@casks
}{%
  \endarray\right.%
}
\def\env@casks{%
  \let\@ifnextchar\new@ifnextchar
  \left\lbrack
  \def\arraystretch{1.2}%
  \array{@{}l@{\quad}l@{}}%
}

\newcounter{numt}
\newcounter{col}
\newcounter{coll}

\makeatother

\renewcommand{\ge}{\geqslant}
\renewcommand{\le}{\leqslant}

\newcommand{\eqi}{\equiv}

\newcommand{\subs}{\subset}

\newcommand{\wo}{\backslash}
\newcommand{\cln}{\colon}

\newcommand{\hra}{\hookrightarrow}

\newcommand{\wt}{\widetilde}

\newcommand*{\bw}[1]{#1\nobreak\discretionary{}{\hbox{$\mathsurround=0pt #1$}}{}}

\newcommand{\br}[1]{\bigl(#1\bigr)}

\newcommand{\bc}[1]{\bigl\{#1\bigr\}}

\newcommand{\Bbl}{\Bigm\wo}

\newcommand{\mbb}{\mathbb}
\newcommand{\mbf}{\mathbf}
\newcommand{\mcl}{\mathcal}

\newcommand{\R}{\mbb{R}}

\newcommand{\T}{\mbb{T}}
\newcommand{\F}{\mbb{F}}
\newcommand{\Cbb}{\mbb{C}}

\newcommand{\Pc}{\mcl{P}}

\newcommand{\ga}{\gamma}

\newcommand{\la}{\lambda}

\newcommand{\ta}{\theta}
\newcommand{\ph}{\varphi}

\DeclareMathOperator{\Hom}{Hom}

\DeclareMathOperator{\End}{End}

\DeclareMathOperator{\rk}{rk}

\newcommand{\GL}{\mbf{GL}}
\newcommand{\SL}{\mbf{SL}}
\newcommand{\Or}{\mbf{O}}
\newcommand{\SO}{\mbf{SO}}

\newcommand{\bom}{\boldmath}

\begin{document}

\author{O.\,G.\?Styrt}
\title{The existence\\
of a~semialgebraic continuous\\
factorization map\\
for some compact linear groups}
\date{}
\newcommand{\udk}{512.815.1+512.815.6+512.816.1+512.816.2}
\newcommand{\msc}{14L24+22C05}

\maketitle

{\leftskip\parind\rightskip\parind
It is proved that each of compact linear groups of one special type admits a~semialgebraic continuous factorization map onto a~real vector space.

\smallskip

\textbf{Key words\:} Lie group, factorization map of an action, semialgebraic map.\par}

\section{Introduction}\label{introd}

In this paper, we prove that each of compact linear groups of one certain type admits a~semialgebraic continuous factorization map onto a~real vector
space. This problem arose from the question when the topological quotient of a~compact linear group is homeomorphic to a~vector space that was researched
in \cite{MAMich,My1,My2,My3,My4}.

For convenience, denote by~$\Pc_{\R}$ (resp. by~$\Pc_{\Cbb}$) the class of all finite-dimensional Euclidean (resp. Hermitian) spaces. Let~$\F$ be one of
the fields $\R$ and~$\Cbb$. We will write $\F^k$ ($k\ge0$) for the space $\F^k\in\Pc_{\F}$ whose standard basis is orthonormal. For any spaces
$W_1,W_2\in\Pc_{\F}$, we will write $W_1\oplus W_2$ for their \textit{orthogonal} direct sum $W_1\oplus W_2\in\Pc_{\F}$ (unless they are linearly
independent subspaces of the same space). For any space $W\in\Pc_{\F}$, denote by $S(W)$ the real subspace
$\bc{A\bw\in\End_{\F}(W)\cln A=A^*}\bw\subs\End_{\F}(W)$, by $\Or(W)$ the compact Lie group $\bc{C\bw\in\GL_{\F}(W)\cln CC^*\bw=E}\bw\subs\GL_{\F}(W)$,
and by $\SO(W)$ its compact subgroup $\Or(W)\cap\SL_{\F}(W)$.

The following theorem is the main result of the paper and will be proved in \Ss\ref{promain}.

\begin{theorem}\label{main} Consider a~number $k\in\{1,2\}$, a~space $W\in\Pc_{\F}$ such that $\dim_{\F}W\ge2-k$, and the representation
\eqn{\label{repm}
\Or(W)\cln\br{S(W)/(\R E)}\oplus\Hom_{\F}(\F^k,W),\,C\cln(A+\R E,B)\to(CAC^{-1}+\R E,CB).}
\begin{nums}{-1}
\item If $k=1$, then there exist a~real vector space~$V$ and, in terms of the representation
$\Or(W)\cln V\oplus\F,\,C\cln(v,\la)\to(v,\det_{\F}C\cdot\la)$, a~semialgebraic continuous $\br{\Or(W)}$\д equivariant surjective map
$\br{S(W)/(\R E)}\oplus\Hom_{\F}(\F^k,W)\to V\oplus\F$ whose fibres coincide with the $\br{\SO(W)}$\д orbits.
\item If $k=2$, then there exist a~real vector space~$V$ and a~semialgebraic continuous surjective map $\br{S(W)/(\R E)}\oplus\Hom_{\F}(\F^k,W)\to V$
whose fibres coincide with the $\br{\Or(W)}$\д orbits.
\end{nums}
\end{theorem}

\section{Auxiliary facts}\label{facts}

Consider a~space $W\in\Pc_{\F}$ and the subsets $S_+(W):=\bc{A\bw\in S(W)\cln A\bw\ge0}\subs S(W)$ and
$S_0(W):=\br{S_+(W)}\Bbl\br{\GL_{\F}(W)}\subs S_+(W)\subs S(W)$.

Set $M(W):=\bc{(A,\la)\in S_+(W)\times\F\cln\det_{\F}A\bw=|\la|^2}\subs S(W)\oplus\F$.

The following statement is well-known.

\begin{stm}\label{quot} The fibres of the map $\pi\cln\End_{\F}(W)\to S(W)\oplus\F,\,X\to(XX^*,\det_{\F}X)$ coincide with the orbits of the action
\eqn{\label{repr}
\SO(W)\cln\End_{\F}(W),\,C\cln X\to XC^{-1},}
and $\pi\br{\End_{\F}(W)}=M(W)\subs S(W)\oplus\F$.
\end{stm}

\begin{lemma}\label{bij} The map $M(W)\to\br{S(W)/(\R E)}\oplus\F,\,(A,\la)\to(A+\R E,\la)$ is a~bijection.
\end{lemma}

\begin{proof} See Lemma~5.1 in~\cite{My2}.
\end{proof}

\begin{imp}\label{bi} The map $S_0(W)\to S(W)/(\R E),\,A\to A+\R E$ is a~bijection.
\end{imp}

\begin{imp}\label{sred} The map $\End_{\F}(W)\bw\to\br{S(W)/(\R E)}\bw\oplus\F,\,X\to(XX^*+\R E,\det_{\F}X)$ is surjective, and its fibres coincide with
the orbits of the action~\eqref{repr}.
\end{imp}

\begin{proof} Follows from Statement~\ref{quot} and Lemma~\ref{bij}.
\end{proof}

Now consider arbitrary spaces $W,W_1\in\Pc_{\F}$ over the filed~$\F$.

\begin{stm}\label{quo} The fibres of the map $\pi\cln\Hom_{\F}(W_1,W)\to S(W),\,X\to XX^*$ coincide with the orbits of the action
\eqn{\label{rep}
\Or(W_1)\cln\Hom_{\F}(W_1,W),\,C_1\cln X\to XC_1^{-1},}
and $\pi\br{\Hom_{\F}(W_1,W)}=\bc{A\bw\in S_+(W)\cln\rk_{\F}A\le\dim_{\F}W_1}\subs S(W)$. In particular, in the case $\dim_{\F}W_1\bw=\dim_{\F}W\bw-1$,
we have $\pi\br{\Hom_{\F}(W_1,W)}=S_0(W)\subs S(W)$.
\end{stm}

We omit the proof since it is clear.

\begin{imp}\label{sre} Suppose that $\dim_{\F}W_1\bw=\dim_{\F}W\bw-1$. Then the map $\Hom_{\F}(W_1,W)\bw\to S(W)/(\R E),\,X\bw\to XX^*\bw+\R E$ is
surjective, and its fibres coincide with the orbits of the action~\eqref{rep}.
\end{imp}

\begin{proof} Follows from Statement~\ref{quo} and Corollary~\ref{bi}.
\end{proof}

\section{Proof of the main result}\label{promain}

In this section, we will prove Theorem~\ref{main}.

First of all, let us prove the claim for a~pair $(k,W)$ ($k\in\{1,2\}$, $W\in\Pc_{\F}$, $\dim_{\F}W=2-k$).

If $k=2$ and $W=0$, then the space of the representation~\eqref{repm} is trivial.
If $k=1$ and $\dim_{\F}W=1$, then the representation~\eqref{repm} is isomorphic to the tautological representation of the group $\Or(W)$ and,
consequently, to the representation $\Or(W)\cln\F,\,C\cln\la\to\det_{\F}C\cdot\la$, and the subgroup $\SO(W)\subs\Or(W)$ is trivial.

This completely proves the claim for a~pair $(k,W)$ ($k\in\{1,2\}$, $W\in\Pc_{\F}$, $\dim_{\F}W=2-k$).

Take an arbitrary pair $(k,W)$ ($k\in\{1,2\}$, $W\in\Pc_{\F}$, $n:=\dim_{\F}W>2-k$) and assume that the claim of Theorem~\ref{main} holds for any pair
$(k,W_1)$ ($W_1\in\Pc_{\F}$, $\dim_{\F}W_1=n-1$). We will now prove it for the pair $(k,W)$.

Since $n>2-k\ge0$, there exists a~space $W_1\in\Pc_{\F}$ such that $\dim_{\F}W_1\bw=n\bw-1$. Consider the embedding
$R\cln\Or(W_1)\hra\Or(W_1\oplus\F^k),\,R(C_1)\cln(x_1,x_2)\to(C_1x_1,x_2)$. The representations
\eqn{\label{adj}
\Or(W_1)\cln S(W_1\oplus\F^k),\,C_1\cln A\to R(C_1)AR(C_1^{-1})}
and $\Or(W_1)\cln S(W_1)\oplus\Hom_{\F}(\F^k,W_1)\oplus S(\F^k),\,C_1\cln(A_1,B,A_2)\to(C_1A_1C_1^{-1},C_1B,A_2)$ are isomorphic via the
$\br{\Or(W_1)}$\д equivariant $\R$\д linear isomorphism
\eqac{
\ta\cln S(W_1)\oplus\Hom_{\F}(\F^k,W_1)\oplus S(\F^k)\to S(W_1\oplus\F^k),\\
\ta(A_1,B,A_2)\cln(x_1,x_2)\to(A_1x_1+Bx_2,B^*x_1+A_2x_2).}
Therefore, the representation~\eqref{adj} is isomorphic to the direct sum of some trivial real representation of the group $\Or(W_1)$ and the
representation
\equ{
\Or(W_1)\cln\br{S(W_1)/(\R E)}\oplus\Hom_{\F}(\F^k,W_1),\,C_1\cln(A_1+\R E,B)\to(C_1A_1C_1^{-1}+\R E,C_1B)}
with trivial fixed point subspace. Thus, there exist a~real vector space~$V'$ and an isomorphism
$\ph\cln S(W_1\oplus\F^k)/(\R E)\to V'\oplus\br{S(W_1)/(\R E)}\oplus\Hom_{\F}(\F^k,W_1)$ of the representations
\eqac{
\Or(W_1)\cln S(W_1\oplus\F^k)/(\R E),\,C_1\cln A+\R E\to R(C_1)AR(C_1^{-1})+\R E;\\
\Or(W_1)\cln V'\oplus\br{S(W_1)/(\R E)}\oplus\Hom_{\F}(\F^k,W_1),\\
C_1\cln(v',A_1+\R E,B)\to(v',C_1A_1C_1^{-1}+\R E,C_1B).}
In terms of the representations
\eqac{
\Or(W)\times\Or(W_1)\cln\Hom_{\F}(W_1\oplus\F^k,W),\,(C,C_1)\cln Z\to CZR(C_1^{-1});\\
\Or(W)\times\Or(W_1)\cln\br{S(W)/(\R E)}\oplus\Hom_{\F}(\F^k,W),\\
(C,C_1)\cln(A+\R E,B)\to(CAC^{-1}+\R E,CB);\\
\Or(W)\times\Or(W_1)\cln V'\oplus\br{S(W_1)/(\R E)}\oplus\Hom_{\F}(\F^k,W_1),\\
(C,C_1)\cln(v',A_1+\R E,B)\to(v',C_1A_1C_1^{-1}+\R E,C_1B),}
the maps
\eqac{
\pi_0\cln\Hom_{\F}(W_1\oplus\F^k,W)\to\br{S(W)/(\R E)}\oplus\Hom_{\F}(\F^k,W),\\
Z\to\br{(Z|_{W_1})(Z|_{W_1})^*+\R E,Z|_{\F^k}};\\
\psi\cln\Hom_{\F}(W_1\oplus\F^k,W)\to V'\oplus\br{S(W_1)/(\R E)}\oplus\Hom_{\F}(\F^k,W_1),\,Z\to\ph(Z^*Z+\R E)}
are $\br{\Or(W)\times\Or(W_1)}$\д equivariant. By Corollary~\ref{sre}, the map~$\pi_0$ is surjective, and its fibres coincide with the
$\br{\Or(W_1)}$\д orbits.

\begin{cas} $k=1$.
\end{cas}

Consider the representation
\eqac{
\Or(W)\times\Or(W_1)\cln V'\oplus\br{S(W_1)/(\R E)}\oplus\Hom_{\F}(\F^k,W_1)\oplus\F,\\
(C,C_1)\cln(v',A_1+\R E,B,\mu)\to(v',C_1A_1C_1^{-1}+\R E,C_1B,\det_{\F}C\cdot\det_{\F}C_1^{-1}\cdot\mu).}
Since $\dim_{\F}W_1\bw=n\bw-1$, there exist a~real vector space~$V_1$ and, in terms of the representation
$\Or(W)\times\Or(W_1)\cln V'\oplus V_1\oplus\F^2,\,(C,C_1)\cln(v',v_1,\la,\mu)\to(v',v_1,\det_{\F}C_1\cdot\la,\det_{\F}C\cdot\det_{\F}C_1^{-1}\cdot\mu)$,
a~semialgebraic continuous $\br{\Or(W)\times\Or(W_1)}$\д equivariant surjective map
\equ{
\pi_1\cln V'\oplus\br{S(W_1)/(\R E)}\oplus\Hom_{\F}(\F^k,W_1)\oplus\F\to V'\oplus V_1\oplus\F^2}
whose fibres coincide with the $\br{\SO(W_1)}$\д orbits. Since $\dim_{\F}(W_1\oplus\F^k)\bw=(n\bw-1)\bw+k\bw=n\bw=\dim_{\F}W$, we can identify $W$ and
$W_1\oplus\F^k$ as spaces of the class~$\Pc_{\F}$. The map
\eqac{
\wt{\psi}\cln\Hom_{\F}(W_1\oplus\F^k,W)\to V'\oplus\br{S(W_1)/(\R E)}\oplus\Hom_{\F}(\F^k,W_1)\oplus\F,\\
Z\to\br{\psi(Z),\det_{\F}Z}=\br{\ph(Z^*Z+\R E),\det_{\F}Z}}
is $\br{\Or(W)\times\Or(W_1)}$\д equivariant. Also, by Corollary~\ref{sred}, the map~$\wt{\psi}$ is surjective, and its fibres coincide with the
$\br{\SO(W)}$\д orbits. Hence, the map $\pi_1\circ\wt{\psi}\cln\Hom_{\F}(W_1\oplus\F^k,W)\bw\to V'\bw\oplus V_1\bw\oplus\F^2$ is
$\br{\Or(W)\times\Or(W_1)}$\д equivariant and surjective, and its fibres coincide with the $\br{\SO(W)\times\SO(W_1)}$\д orbits. If we consider the
representation
\eqac{
\Or(W)\times\Or(W_1)\cln V'\oplus V_1\oplus\R\oplus\F,\,(C,C_1)\cln(v',v_1,t,\la)\to(v',v_1,t,\det_{\F}C\cdot\la),}
then the map $\ga\cln V'\oplus V_1\oplus\F^2\to V'\oplus V_1\oplus\R\oplus\F,\,(v',v_1,\la,\mu)\to\br{v',v_1,|\la|^2-|\mu|^2,\la\mu}$ is
$\br{\Or(W)\times\Or(W_1)}$\д equivariant. Denote by~$\T$ the multiplicative group $\bc{c\in\F\cln|c|=1}$. We have $n>2-k=1$, $\dim_{\F}W_1\bw=n\bw-1>0$.
Consequently,
\eqn{\label{sot}
\begin{array}{c}
\bc{\det_{\F}C_1\cln C_1\in\Or(W_1)}=\T.
\end{array}}
One can easily see that the map $\F^2\to\R\oplus\F,\,(\la,\mu)\to\br{|\la|^2-|\mu|^2,\la\mu}$ is surjective, and its fibres coincide with the orbits of
the representation $\T\cln\F^2,\,c\cln(\la,\mu)\to(c\la,c^{-1}\mu)$. By~\eqref{sot}, the map~$\ga$ is surjective, and its fibres coincide with the
$\br{\Or(W_1)}$\д orbits. It follows from above that the map
$\ga\circ\pi_1\circ\wt{\psi}\cln\Hom_{\F}(W_1\oplus\F^k,W)\bw\to V'\oplus V_1\oplus\R\oplus\F$ is $\br{\Or(W)\times\Or(W_1)}$\д equivariant and
surjective, and its fibres coincide with the $\br{\SO(W)\times\Or(W_1)}$\д orbits.

Recall that the map~$\pi_0$ is $\br{\Or(W)\times\Or(W_1)}$\д equivariant and surjective, and its fibres coincide with the $\br{\Or(W_1)}$\д orbits.
Therefore, there exists an $\br{\Or(W)}$\д equivariant surjective map $\pi\cln\br{S(W)/(\R E)}\oplus\Hom_{\F}(\F^k,W)\to V'\oplus V_1\oplus\R\oplus\F$
satisfying $\pi\circ\pi_0\eqi\ga\circ\pi_1\circ\wt{\psi}$ whose fibres coincide with the $\br{\SO(W)}$\д orbits.
Since the maps $\ga,\pi_1,\wt{\psi},\pi_0$ are semialgebraic, continuous, and surjective, so is the map~$\pi$.

This completely proves the claim in the case $k=1$.

\begin{cas} $k=2$.
\end{cas}

Since $\dim_{\F}W_1\bw=n\bw-1$, there exist a~real vector space~$V_1$ and a~semialgebraic continuous surjective map
$\pi_1\cln V'\oplus\br{S(W_1)/(\R E)}\oplus\Hom_{\F}(\F^k,W_1)\to V'\oplus V_1$ whose fibres coincide with the $\br{\Or(W_1)}$\д orbits. Further,
$\dim_{\F}(W_1\oplus\F^k)\bw=(n\bw-1)\bw+k\bw=n\bw+1\bw=\dim_{\F}W\bw+1$, and, by Corollary~\ref{sre}, the map~$\psi$ is surjective, and its fibres
coincide with the $\br{\Or(W)}$\д orbits. Since the map~$\psi$ is $\br{\Or(W)\times\Or(W_1)}$\д equivariant, the map
$\pi_1\circ\psi\cln\Hom_{\F}(W_1\oplus\F^k,W)\to V'\oplus V_1$ is surjective, and its fibres coincide with the $\br{\Or(W)\times\Or(W_1)}$\д orbits.
Recall that the map~$\pi_0$ is $\br{\Or(W)\times\Or(W_1)}$\д equivariant and surjective, and its fibres coincide with the $\br{\Or(W_1)}$\д orbits.
Hence, there exists a~surjective map $\pi\cln\br{S(W)/(\R E)}\oplus\Hom_{\F}(\F^k,W)\to V'\oplus V_1$ satisfying $\pi\circ\pi_0\eqi\pi_1\circ\psi$ whose
fibres coincide with the $\br{\Or(W)}$\д orbits. Since the maps $\pi_1,\psi,\pi_0$ are semialgebraic, continuous, and surjective, so is the map~$\pi$.

This completely proves the claim in the case $k=2$.

Thus, we have completely proved Theorem~\ref{main} (using mathematical induction on $\dim_{\F}W$ separately for each number $k\in\{1,2\}$).


\newpage

\begin{thebibliography}{9}
\bibitem{MAMich}
M.\,A.\?Mikhailova, \textit{On the quotient space modulo the action of a finite group generated by pseudoreflections}, Mathematics of the USSR-Izvestiya,
1985, vol.\,24, \No1, 99---119.
\bibitem{My1}
O.\,G.\?Styrt, \textit{On the orbit space of a compact linear Lie group with commutative connected component}, Tr. Mosk. Mat. O-va, 2009, vol.\,70,
235---287 (Russian).
\bibitem{My2}
O.\,G.\?Styrt, \textit{On the orbit space of a three-dimensional compact linear Lie group}, Izv. RAN, Ser. math., 2011, vol.\,75, \No4, 165---188
(Russian).
\bibitem{My3}
O.\,G.\?Styrt, \textit{On the orbit space of an irreducible representation of a special unitary group}, Tr. Mosk. Mat. O-va, 2013, vol.\,74, \No1,
175---199 (Russian).
\bibitem{My4}
O.\,G.\?Styrt, \textit{On the orbit spaces of irreducible representations of simple compact Lie groups of types $B$, $C$, and~$D$}, J.~Algebra, 2014,
vol.\,415, 137---161.
\end{thebibliography}
\end{document}